\documentclass[12pt,twoside]{amsart}
\usepackage{amsmath, amsthm, amscd, amsfonts, amssymb, graphicx}
\usepackage{enumerate}
\usepackage[colorlinks=true,
linkcolor=blue,
urlcolor=cyan,
citecolor=red]{hyperref}
\usepackage{mathrsfs}
\newtheorem{theorem}{Theorem}[section]
\newtheorem{lemma}{Lemma}[section]

\newtheorem{corollary}{Corollary}[section]

\newtheorem{proposition}{Proposition}[section]
\numberwithin{equation}{section}

\begin{document}
\title{New inequalities for  the generalized Karcher mean}
\author{Mohammed Sababheh, Hamid Reza Moradi and Zahra Heydarbeygi}
 \maketitle

%------------------------------------------------------------------------------------%
\pagestyle{myheadings}
\markboth{\centerline {New inequalities for  the generalized Karcher mean}}
{\centerline {M. Sababheh, H.R. Moradi \& Z. Heydarbeygi}}
\bigskip
\bigskip
%------------------------------------------------------------------------------------%
%------------------------------------------------------------------------------------

\begin{abstract}
Recently, P\'{a}lfia introduced a generalized Karcher mean as a solution of an operator equation. In this article, we present several relations for this new mean.

In particular, we investigate the behavior of this generalized mean when filtered through operator monotone functions and positive linear maps; revealing its information monotonicity.
\end{abstract}

\section{Introduction}
Let $\mathcal{M}_k$ be the algebra of all complex $k\times k$ matrices. A matrix $A\in \mathcal{M}_k$ is said to be positive definite, and is written $A>0$, if $\left<Ax,x\right> >0$ for all nonzero vectors $x\in\mathbb{C}^k.$ If $A,B\in\mathcal{M}_k$ are Hermitian matrices such that $A-B$ is positive definite, we write $A>B.$ If $A-B$ is positive semidefinite, we write $A\geq B.$ The identity matrix in $\mathcal{M}_k$ will be denoted by $I_k$, or $I$ if no confusion arises. A linear map $\Phi:\mathcal{M}_{k_1}\to\mathcal{M}_{k_2}$ is said to be positive if $\Phi(A)\geq 0$ whenever $A\geq 0.$ If, in addition, $\Phi(I_{k_1})=I_{k_2},$ we say that $\Phi$ is a unital positive  map.

A norm $|||\cdot|||$ on ${{\mathcal{M}}_{k}}$  is said to be unitarily invariant if $|||UXV|||=|||X|||$  for all $X\in {{\mathcal{M}}_{k}}$  and all unitary matrices $U,V\in {{\mathcal{M}}_{k}}$. We denote by $\left\| A \right\|$ the spectral (operator) norm of $A$, i.e., $\left\| A \right\|=\max \left\{ \left\| Ax \right\|:\text{ }\left\| x \right\|=1,~x\in {{\mathbb{C}}^{k}} \right\}$.

If $f:[0,\infty)\to \mathbb{R}$ is a continuous function and $A\geq 0,$ we compute $f(A)$ by functional calculus. That is, if $A=U{\text{diag}}(\lambda_i)U^*$ is the spectral decomposition of $A$, $f(A)$ is defined via the formula $f(A)=U{\text{diag}}(f(\lambda_i))U^*$.

For a positive weight vector $w=\left( {{w}_{1}},\ldots ,{{w}_{n}} \right)$ (i.e., $w_i>0$ and $\sum_{i=1}^{n}w_i=1$) and positive definite matrices $\mathbb{A}=({{A}_{1}},\ldots ,{{A}_{n}})$, the Karcher mean $\Lambda(w;\mathbb{A})$ is the unique positive solution of
\begin{equation}\label{karcher}
\sum\limits_{i=1}^{n}{{{w}_{i}}\log \left( {{X}^{-\frac{1}{2}}}{{A}_{i}}{{X}^{-\frac{1}{2}}} \right)}=0.
\end{equation}
We call \eqref{karcher} the Karcher equation (see \cite{moakher}).

In \cite{3}, Lim and P\'alfia  introduced the notion of matrix power mean of positive definite matrices of some
fixed dimension. The matrix power mean ${{P}_{t}}\left( w,\mathbb{A} \right)$ is defined by the unique positive definite
solution of the following non-linear equation:
\begin{equation}\label{4}
X=\sum\limits_{i=1}^{n}{{{w}_{i}}\left( X{{\sharp}_{t}}{{A}_{i}} \right)},\quad\text{ }t\in \left( 0,1 \right]
\end{equation}
where $A{{\sharp}_{t}}B={{A}^{{1}/{2}\;}}{{\left( {{A}^{{-1}/{2}\;}}B{{A}^{{-1}/{2}\;}} \right)}^{t}}{{A}^{{1}/{2}\;}}$ is the $t$-weighted geometric mean of $A$ and $B$. For $t\in \left[ -1,0 \right)$, it is defined by ${{P}_{t}}\left( w;\mathbb{A} \right)={{P}_{-t}}{{\left( w;{{\mathbb{A}}^{-1}} \right)}^{-1}}$, where ${{\mathbb{A}}^{-1}}=\left( A_{1}^{-1},\ldots ,A_{n}^{-1} \right)$. As $t\to 0$, the power mean $P_t$ coincides with the Karcher mean $\Lambda$.

P\'alfia \cite{palfia} generalized the operator equation \eqref{4} to the following form
\begin{equation}\label{5}
\sum\limits_{i=1}^{n}{{{w}_{i}}\;g\left( {{X}^{-\frac{1}{2}}}{{A}_{i}}{{X}^{-\frac{1}{2}}} \right)}=0
\end{equation}
where $w$ is a weight vector and $g$ is an operator monotone function on $\left( 0,\infty  \right)$ with 
$g\left( 1 \right)=0$ and $g'\left( 1 \right)=1$.

Of course, the Karcher and the power means can be obtained by setting $g\left( x \right)=\log x$ and $g\left( x \right)=\frac{{{x}^{t}}-1}{t}$ in \eqref{5}, respectively. In what follows ${\sigma }_{g}\left( w;\mathbb{A} \right)$ denotes
the solution $X$ of \eqref{5}.

Let $\mathscr{M}$ denote the set of all operator monotone functions on $(0,\infty)$, and let
\[\mathscr{L}=\left\{ g\in \mathscr{M}|g\left( 1 \right)=0\text{ and }g'\left( 1 \right)=1 \right\}.\]
In \cite[Proposition 6.15]{palfia}, the author showed the following order among ${{\sigma }_{g}}\left( w;\mathbb{A} \right)$, weighted harmonic and arithmetic means 
\begin{equation}\label{6}
{{\left( \sum\limits_{i=1}^{n}{{{w}_{i}}A_{i}^{-1}} \right)}^{-1}}\le {{\sigma }_{g}}\left( w;\mathbb{A} \right)\le \sum\limits_{i=1}^{n}{{{w}_{i}}{{A}_{i}}}.
\end{equation}
The proof of \eqref{6} is based on the observation that when $g\in \mathscr{L}$, it follows that $1-x^{-1}\leq g(x)\leq x-1.$

Recently in \cite{yamazaki},  the following extension of Ando-Hiai inequality has been shown: Let $g\in \mathscr{L}$, $\mathbb{A}=\left( {{A}_{1}},\ldots ,{{A}_{n}} \right)$ be an $n$-tuple of positive definite matrices and $w=\left( {{w}_{1}},\ldots ,{{w}_{n}} \right)$ be a weight vector. 
Then the implication
\begin{equation}\label{yamaz_ineq_p}
{{\sigma }_{g}}\left( w;\mathbb{A} \right)\le I\Rightarrow {{\sigma }_{{{g}_{p}}}}\left( w;{{\mathbb{A}}^{p}} \right)\le I
\end{equation}
 holds for all $p\ge 1$, where ${{g}_{p}}\left( x \right)=pg\left( {{x}^{\frac{1}{p}}} \right)$.

Throughout this paper we assume that $g\in \mathscr{L}$. Our target in this article is to present generalizations and counterparts of \eqref{6} and \eqref{yamaz_ineq_p}   via Kantorovich constant $K\left( h,2 \right)=\frac{{{\left( h+1 \right)}^{2}}}{4h}$.  In applications, we give an analogous result of \cite{gumus} for $n$-tuple of positive definite matrices. Our results extend the results appearing in \cite{4} to the context of the solution of the generalized Karcher equation (GKE) and present new generalizations that reflect the behavior of these means under positive linear maps and operator monotone functions.

Further, we present a natural extension of the inequality \cite{bourin}
\begin{equation}\label{bourin_geo_ineq}
\left<(A\sharp_v B)x,x\right>\leq \left<Ax,x\right>\sharp_v\left<Bx,x\right>
\end{equation}
valid for all unit vectors $x\in\mathbb{C}^k,$  $0\leq v\leq 1$ and positive definite matrices $A,B\in\mathcal{M}_k$. Many other results generalizing the action of operator monotone functions on two matrices will be presented too.

\section{Reverses of \eqref{6} and their refinements }

In this section we present the reversed versions of \eqref{6} first, then we prove refinements using the well known Kantorovich inequality and its refinement.\\
First, we have the following useful lemma \cite[Corollary 2.36]{pe}.
\begin{lemma}
Let $\left( {{A}_{1}},\ldots ,{{A}_{n}} \right)$ be an $n$-tuple of positive definite matrices in $\mathcal{M}_k$ with $mI\le {{A}_{i}}\le MI\left( i=1,\ldots ,n \right)$ for some scalars $0<m<M$, and $w=\left( {{w}_{1}},\ldots ,{{w}_{n}} \right)$ be a  weight vector. If $f,g:[m,M]\to (0,\infty)$ are convex, then
\begin{equation}\label{8}
\sum\limits_{i=1}^{n}{{{w}_{i}}f\left( {{A}_{i}} \right)}\le \mu \left( m,M,f,g \right)g\left( \sum\limits_{i=1}^{n}{{{w}_{i}}{{A}_{i}}} \right)
\end{equation}
where
\[\mu \left( m,M,f,g \right)=\max \left\{ \frac{1}{g\left( t \right)}\left( \frac{M-t}{M-m}f\left( m \right)+\frac{t-m}{M-m}f\left( M \right) \right):~ m\le t\le M \right\}.\]
In particular,
\begin{equation}\label{7}
\sum_{i=1}^{n}w_if(A_i)\leq \mu \left( m,M,f \right)f\left(\sum_{i=1}^{n}w_i A_i\right),
\end{equation}
where
\[\mu \left( m,M,f \right)= \max \left\{ \frac{1}{f\left( t \right)}\left( \frac{M-t}{M-m}f\left( m \right)+\frac{t-m}{M-m}f\left( M \right) \right):~ m\le t\le M \right\}.\]
\end{lemma}
Since $f(t)=t^{-1}$ is convex on $[m,M]$, inequality \eqref{7} can be applied. For this function, direct calculations show that  (see the proof of Theorem 2 in \cite{fujii})
\[\max \left\{ \frac{1}{{{t}^{-1}}}\left( \frac{M-t}{M-m}{{m}^{-1}}+\frac{t-m}{M-m}{{M}^{-1}} \right):~ m\le t\le M \right\}=\frac{{{\left( M+m \right)}^{2}}}{4Mm}.\]

Consequently,
\begin{equation}\label{inverse_needed}
\left(\sum_{i=1}^{n}w_i A_i^{-1} \right)\leq \frac{{{\left( M+m \right)}^{2}}}{4Mm} \left(\sum_{i=1}^{n}w_iA_i\right)^{-1}.
\end{equation}

The quantity $\frac{{{\left( M+m \right)}^{2}}}{4Mm}$ is well known as the Kantorovich constant and is denoted by $K(h,2),$ where $K(h,2)=\frac{(h+1)^2}{4h}$ and $h=\frac{M}{m}.$ This notation will be used in the sequel.

Notice that replacing $A_i$ by $A_i^{-1}$ in \eqref{inverse_needed} and $[m,M]$ by $\left[\frac{1}{M},\frac{1}{m}\right]$ imply
\begin{equation}\label{inverse_needed_2}
\left(\sum_{i=1}^{n}w_i A_i \right)\leq K(h,2) \left(\sum_{i=1}^{n}w_iA_i^{-1}\right)^{-1}, \text{ } h=\frac{M}{m}.
\end{equation}
\begin{proposition}\label{02}
	Let $\mathbb{A}=\left( {{A}_{1}},\ldots ,{{A}_{n}} \right)$ be an $n$-tuple of positive definite matrices in $\mathcal{M}_k$ with $mI\le {{A}_{i}}\le MI\left( i=1,\ldots ,n \right)$ for some scalars $0<m<M,$ and $w=\left( {{w}_{1}},\ldots ,{{w}_{n}} \right)$ be a  weight vector. Then
	\begin{equation}\label{03}
	\sum\limits_{i=1}^{n}{{{w}_{i}}{{A}_{i}}}\le K\left( h,2 \right){{\sigma }_{g}}\left( w;\mathbb{A} \right),
	\end{equation}
	and
	\begin{equation}\label{04}
	{{\sigma }_{g}}\left( w;\mathbb{A} \right)\le K\left( h,2 \right){{\left( \sum\limits_{i=1}^{n}{{{w}_{i}}A_{i}^{-1}} \right)}^{-1}}.
	\end{equation}
\end{proposition}
\begin{proof}
	Using  \eqref{inverse_needed_2}, then \eqref{6} imply
	\begin{align}
	\sum\limits_{i=1}^{n}{{{w}_{i}}{{A}_{i}}}&\leq  K\left( h,2 \right){{\left( \sum\limits_{i=1}^{n}{{{w}_{i}}A_{i}^{-1}} \right)}^{-1}} \label{05}\\ 
	& \le K\left( h,2 \right){{\sigma }_{g}}\left( w;\mathbb{A} \right).  \nonumber
	\end{align}
 This proves \eqref{03}.
	
	The inequality \eqref{04} follows from RHS of \eqref{6} and the inequality \eqref{05}.
\end{proof}

Next, we deduce refinements of both inequalities \eqref{03} and \eqref{04}. 

\begin{proposition}\label{18}
	Let $\mathbb{A}=\left( {{A}_{1}},\ldots ,{{A}_{n}} \right)$ be a $n$-tuple of positive definite matrices with $mI\le {{A}_{i}}\le MI\left( i=1,\ldots ,n \right)$ for some scalars $0<m<M$, and $w=\left( {{w}_{1}},\ldots ,{{w}_{n}} \right)$ be a weight vector. Then
\begin{equation}\label{ineq1}
\sum\limits_{i=1}^{n}{{{w}_{i}}{{A}_{i}}}\le \sum\limits_{i=1}^{n}{{{w}_{i}}{{M}^{\frac{{{m}^{-1}}I-A_{i}^{-1}}{{{m}^{-1}}-{{M}^{-1}}}}}}{{m}^{\frac{A_{i}^{-1}-{{M}^{-1}I}}{{{m}^{-1}}-{{M}^{-1}}}}}\le K\left( h,2 \right){{\sigma }_{g}}\left( w;\mathbb{A} \right),
\end{equation}	
	and
	\begin{equation}\label{ineq2}
	{{\sigma }_{g}}\left( w;\mathbb{A} \right)\le \sum\limits_{i=1}^{n}{{{w}_{i}}{{M}^{\frac{{{m}^{-1}}I-A_{i}^{-1}}{{{m}^{-1}}-{{M}^{-1}}}}}}{{m}^{\frac{A_{i}^{-1}-{{M}^{-1}I}}{{{m}^{-1}}-{{M}^{-1}}}}}\le K\left( h,2 \right){{\left( \sum\limits_{i=1}^{n}{{{w}_{i}}A_{i}^{-1}} \right)}^{-1}}.
	\end{equation}
\end{proposition}
\begin{proof}
	Notice that the function $f(t)=t^{-1}$ is log-convex on any positive interval $[m,M].$ Consequently,
\begin{equation}\label{log-convex_m}
f(t)\leq f(m)^{\frac{M-t}{M-m}}f(M)^{\frac{t-m}{M-m}}\leq L(t),
\end{equation}
where 
$$L(t)=\frac{M-t}{M-m}f(m)+\frac{t-m}{M-m}f(M),\text{  } m\leq t\leq M.$$
Now, if $mI\leq A_i\leq MI$ $\left( i=1,\ldots ,n \right)$, we may apply functional calculus in \eqref{log-convex_m} to obtain
$$f(A_i)\leq f(m)^{\frac{MI-A_i}{M-m}}f(M)^{\frac{A_i-mI}{M-m}}\leq L(A_i).$$
Now since $w_i\geq 0,$  the above inequality implies
\begin{equation}\label{log_convex_2}
\sum_{i=1}^{n}w_if(A_i)\leq \sum_{i=1}^{n}w_i f(m)^{\frac{MI-A_i}{M-m}}f(M)^{\frac{A_i-mI}{M-m}}\leq \sum_{i=1}^{n}w_iL(A_i).
\end{equation}
Using \eqref{8}, with $f$ replaced by $L$ and $g$ replaced by $f$, we obtain
$$\sum_{i=1}^{n}w_iL(A_i)\leq \mu(m,M,L,f) f\left(\sum_{i=1}^{n}w_iA_i\right),$$ and hence \eqref{log_convex_2} implies
\begin{align*}
\notag \sum_{i=1}^{n}w_if(A_i)&\leq \sum_{i=1}^{n}w_i f(m)^{\frac{MI-A_i}{M-m}}f(M)^{\frac{A_i-mI}{M-m}}\\
\notag &\leq \mu(m,M,L,f) f\left(\sum_{i=1}^{n}w_iA_i\right).
\end{align*}
Since $f(t)=t^{-1}$ and noting that $L$ and $f$ coincide at $m$ and $M$, \eqref{inverse_needed} implies
\begin{equation}\label{log_convex_4}
\sum_{i=1}^{n}w_iA_i^{-1}\leq \sum_{i=1}^{n}w_im^{\frac{A_i-MI}{M-m}}M^{\frac{mI-A_i}{M-m}}\leq \frac{{{\left( M+m \right)}^{2}}}{4Mm}\left(\sum_{i=1}^{n}w_iA_i\right)^{-1}.
\end{equation}
Replacing $A_i$ by $A_i^{-1}$ and $\left[m,M\right]$ by $\left[\frac{1}{M},\frac{1}{m}\right],$ \eqref{log_convex_4} becomes
\begin{align*}
\sum_{i=1}^{n}w_iA_i\leq \sum\limits_{i=1}^{n}{{{w}_{i}}{{M}^{\frac{{{m}^{-1}}I-A_{i}^{-1}}{{{m}^{-1}}-{{M}^{-1}}}}}}{{m}^{\frac{A_{i}^{-1}-{{M}^{-1}I}}{{{m}^{-1}}-{{M}^{-1}}}}}\leq \frac{{{\left( M+m \right)}^{2}}}{4Mm}\left(\sum_{i=1}^{n}w_iA_i^{-1}\right)^{-1},
\end{align*}
which implies \eqref{ineq1} using \eqref{6}.\\
Now \eqref{ineq2} follows immediately noting \eqref{6} and  \eqref{ineq1}. This completes the proof.	
\end{proof}

In the following, we complement the inequality \eqref{yamaz_ineq_p}.
\begin{theorem}
Let $\mathbb{A}=\left( {{A}_{1}},\ldots ,{{A}_{n}} \right)$ be a $n$-tuple of positive definite matrices in $\mathcal{M}_k$ with $mI\le {{A}_{i}}\le MI\left( i=1,\ldots ,n \right)$ for some scalars $0<m<M$ and $w=\left( {{w}_{1}},\ldots ,{{w}_{n}} \right)$ be a weight vector. Then for all $p\ge 1$ and every unitarily invariant norm $|||\cdot|||$,
\[|||{{\sigma }_{g}}\left( w;{{\mathbb{A}}^{p}} \right)|||\le K\left( m,M,p \right)K{{\left( h,2 \right)}^{p}}|||{{\sigma }_{g}}{{\left( w;\mathbb{A} \right)}^{p}}|||,\] 
where
\[K\left( m,M,p \right)=\frac{m{{M}^{p}}-M{{m}^{p}}}{\left( p-1 \right)\left( M-m \right)}{{\left( \frac{p-1}{p}\frac{{{M}^{p}}-{{m}^{p}}}{m{{M}^{p}}-M{{m}^{p}}} \right)}^{p}}.\]
In particular, if ${{\sigma }_{g}}\left( w;\mathbb{A} \right)\le I$, we have
\begin{equation}\label{Ando.Hiai}
{{\sigma }_{g}}\left( w;{{\mathbb{A}}^{p}} \right)\le K\left( m,M,p \right)K{{\left( h,2 \right)}^{p}}.
\end{equation}
\end{theorem}
\begin{proof}
	We have
	\begin{equation}\label{3}
\begin{aligned}
{{\sigma }_{g}}\left( w;{{\mathbb{A}}^{p}} \right)&\le \sum\limits_{i=1}^{n}{{{w}_{i}}A_{i}^{p}} \\ 
& \le K\left( m,M,p \right){{\left( \sum\limits_{i=1}^{n}{{{w}_{i}}{{A}_{i}}} \right)}^{p}} \\ 
& \le K\left( m,M,p \right)K{{\left( h,2 \right)}^{p}}U{{\sigma }_{g}}{{\left( w;\mathbb{A} \right)}^{p}}{{U}^{*}}  
\end{aligned}	
	\end{equation}
where the first inequality is due to RHS of \eqref{6} and the fact that $mI\le {{A}_{i}}\le MI$ implies ${{m}^{p}}I\le A_{i}^{p}\le {{M}^{p}}I$ ($p>0$), the second one is due to \cite[Remark 4.14]{micic}, and the last inequality follows from \eqref{03} and the fact that for two positive definite matrices $X,Y$ with $X\le Y$ there exists a unitary matrix $U$, such that ${{X}^{p}}\le U{{Y}^{p}}{{U}^{*}}$, for $p>0$, \cite[Exercise III.5.5]{bhatia}. 

One can infer from the above discussion  
	\[|||{{\sigma }_{g}}\left( w;{{\mathbb{A}}^{p}} \right)|||\le K\left( m,M,p \right)K{{\left( h,2 \right)}^{p}}|||{{\sigma }_{g}}{{\left( w;\mathbb{A} \right)}^{p}}|||,\] 
for all $p\ge 1$. Consequently,
\[{{\sigma }_{g}}\left( w;{{\mathbb{A}}^{p}} \right)\le \left\| {{\sigma }_{g}}\left( w;{{A}^{p}} \right) \right\|I\le K\left( m,M,p \right)K{{\left( h,2 \right)}^{p}}\left\| {{\sigma }_{g}}{{\left( w;\mathbb{A} \right)}^{p}} \right\|I,\] where $\left\| \cdot \right\|$ is the spectral (operator) norm. The assumption ${{\sigma }_{g}}\left( w;\mathbb{A} \right)\le I$ implies \eqref{Ando.Hiai}.
\end{proof}

\section{inequalities involving positive linear maps}
In this section we present several relations that describe the behavior of the solution of the GKE under positive linear maps. This study is usually referred to as information monotonicity.  The following lemma is needed to prove our results. 
\begin{lemma}\label{4.1}
\cite{yamazaki} Let $\mathbb{A}=\left( {{A}_{1}},\ldots ,{{A}_{n}} \right)$ be an $n$-tuple of positive definite matrices in $\mathcal{M}_k$ and $w=\left( {{w}_{1}},\ldots ,{{w}_{n}} \right)$ be a weight vector.
\begin{itemize}
	\item[(i)] $\sum\nolimits_{i=1}^{n}{{{w}_{i}}g\left( {{A}_{i}} \right)}\ge 0$ implies ${{\sigma }_{g}}\left( w;\mathbb{A} \right)\ge I$.
	\item[(ii)] ${{\sigma }_{g}}\left( w;{{X}^{*}}\mathbb{A}X \right)={{X}^{*}}{{\sigma }_{g}}\left( w;\mathbb{A} \right)X$ for all invertible $X$.
\end{itemize}
\end{lemma}
Our first result in this direction is the study of information monotonicity of $\sigma_g$. This result extends the corresponding result of \cite{4}, where the power mean $P_t$ was studied. It should be noted here that the first inequality was recently established in \cite[Theorem 6.4 (7)]{palfia}, however, for the reader convenience, we present here a simple proof for it.
\begin{theorem}\label{thm_first_linear_map}
	Let $\mathbb{A}=\left( {{A}_{1}},\ldots ,{{A}_{n}} \right)$ be a $n$-tuple of positive definite matrices in $\mathcal{M}_k$ with $mI\le {{A}_{i}}\le MI\left( i=1,\ldots ,n \right)$ for some scalars $0<m<M$ and $w=\left( {{w}_{1}},\ldots ,{{w}_{n}} \right)$ be a weight vector. Then, for the unital positive linear map $\Phi$,
\[\Phi \left( {{\sigma }_{g}}\left( w;\mathbb{A} \right) \right)\le {{\sigma }_{g}}\left( w;\Phi \left( \mathbb{A} \right) \right)\le K\left( h,2 \right)\Phi \left( {{\sigma }_{g}}\left( w;\mathbb{A} \right) \right).\]
\end{theorem}
\begin{proof}
	For the first inequality, let $X={{\sigma }_{g}}\left( w;\mathbb{A} \right)$. Then 
	\[0=\sum\limits_{i=1}^{n}{{{w}_{i}}g\left( {{X}^{-\frac{1}{2}}}{{A}_{i}}{{X}^{-\frac{1}{2}}} \right)}\text{ }\Rightarrow \text{ }0=\sum\limits_{i=1}^{n}{{{w}_{i}}\left( X{{\sigma }_{g}}{{A}_{i}} \right)},\]
	where $X\sigma_gA_i= {X}^{\frac{1}{2}}g\left( {{X}^{-\frac{1}{2}}}{{A}_{i}}{{X}^{-\frac{1}{2}}} \right){X}^{\frac{1}{2}}.$
	Taking $\Phi $, we obtain 
	\[0=\sum\limits_{i=1}^{n}{{{w}_{i}}\Phi \left( X{{\sigma }_{g}}{{A}_{i}} \right)}\le \sum\limits_{i=1}^{n}{{{w}_{i}}\left( \Phi \left( X \right){{\sigma }_{g}}\Phi \left( {{A}_{i}} \right) \right)}\] 
	where we used the well known Ando's inequality. Whence  
	\[0\le \sum\limits_{i=1}^{n}{{{w}_{i}}g\left( \Phi {{\left( X \right)}^{-\frac{1}{2}}}\Phi \left( {{A}_{i}} \right)\Phi {{\left( X \right)}^{-\frac{1}{2}}} \right)}.\] 
Applying the first part of Lemma \ref{4.1}, we obtain
\begin{align*}
\sum\limits_{i=1}^{n}{{{w}_{i}}g\left( \Phi {{\left( X \right)}^{-\frac{1}{2}}}\Phi \left( {{A}_{i}} \right)\Phi {{\left( X \right)}^{-\frac{1}{2}}} \right)\ge 0}\text{ }&\Rightarrow {{\sigma }_{g}}\left( w;\Phi {{\left( X \right)}^{-\frac{1}{2}}}\Phi \left( \mathbb{A} \right)\Phi {{\left( X \right)}^{-\frac{1}{2}}} \right)\geq I, \\
\end{align*}	
which implies, by the second part of Lemma \ref{4.1},
$$\text{ }\Phi {{\left( X \right)}^{-\frac{1}{2}}}{{\sigma }_{g}}\left( w;\Phi \left( \mathbb{A}\right) \right)\Phi {{\left( X \right)}^{-\frac{1}{2}}}\ge I.$$
	That is
	\[{{\sigma }_{g}}\left( w;\Phi \left( \mathbb{A} \right) \right)\ge \Phi \left( X \right)\] 
	which is equivalent to 
	\[{{\sigma }_{g}}\left( w;\Phi \left( \mathbb{A} \right) \right)\ge \Phi \left( {{\sigma }_{g}}\left( w;\mathbb{A} \right) \right).\] 
	This proves the first desired inequality. For the second inequality, noting the RHS of \eqref{6} and the inequality \eqref{03}, respectively, we obtain
	\[\begin{aligned}
	{{\sigma }_{g}}\left( w;\Phi \left( \mathbb{A} \right) \right)&\le \sum\limits_{i=1}^{n}{{{w}_{i}}\Phi \left( {{A}_{i}} \right)}\\ 
	& =\Phi \left( \sum\limits_{i=1}^{n}{{{w}_{i}}{{A}_{i}}} \right) \\ 
	& \le \Phi \left( K\left( h,2 \right){{\sigma }_{g}}\left( w;\mathbb{A} \right) \right) \\ 
	& =K\left( h,2 \right)\Phi \left( {{\sigma }_{g}}\left( w;\mathbb{A} \right) \right).  
	\end{aligned}\]
	This completes the proof of the second inequality. 
\end{proof}
Now we show how the inequalities in  Proposition \ref{18} could be raised to the $p^{\text{th}}$ power.
\begin{theorem}\label{12}
Let $\mathbb{A}=\left( {{A}_{1}},\ldots ,{{A}_{n}} \right)$ be an $n$-tuple of positive definite matrices in $\mathcal{M}_k$ with $mI\le {{A}_{i}}\le MI\left( i=1,\ldots ,n \right)$ for some scalars $0<m<M$ and $w=\left( {{w}_{1}},\ldots ,{{w}_{n}} \right)$ be a weight vector. If $\Phi :{{\mathcal{M}}_{k_1}}\to {{\mathcal{M}}_{k_2}}$ is a unital positive linear map, then for any $p\ge 2,$
\begin{equation}\label{11}
\Phi {{\left( \sum\limits_{i=1}^{n}{{{w}_{i}}{{M}^{\frac{{{m}^{-1}}I-A_{i}^{-1}}{{{m}^{-1}}-{{M}^{-1}}}}}{{m}^{\frac{A_{i}^{-1}-{{M}^{-1}I}}{{{m}^{-1}}-{{M}^{-1}}}}}} \right)}^{p}}\le {{\left( \frac{{{\left( m+M \right)}^{2}}}{{{4}^{\frac{2}{p}}}{{M}}{{m}}} \right)}^{p}}\Phi {{\left( {{\sigma}_{g}}\left( w;\mathbb{A} \right) \right)}^{p}},
\end{equation}
and

\begin{equation}\label{112nd}
\Phi {{\left( {{\sigma}_{g}}\left( w;\mathbb{A} \right) \right)}^{p}}\le {{\left( \frac{{{\left( m+M \right)}^{2}}}{{{4}^{\frac{2}{p}}}{{M}}{{m}}} \right)}^{p}}\Phi {{\left( \sum\limits_{i=1}^{n}{{{w}_{i}}{{M}^{\frac{{{m}^{-1}}I-A_{i}^{-1}}{{{m}^{-1}}-{{M}^{-1}}}}}{{m}^{\frac{A_{i}^{-1}-{{M}^{-1}I}}{{{m}^{-1}}-{{M}^{-1}}}}}} \right)}^{p}}.
\end{equation}

\end{theorem}
\begin{proof}
First notice that for $t\in \left[ m,M \right]$,
\begin{equation}\label{9}
t+mM{{m}^{\frac{t-M}{M-m}}}{{M}^{\frac{m-t}{M-m}}}=t+{{m}^{\frac{t-m}{M-m}}}{{M}^{\frac{M-t}{M-m}}}\le M+m,
\end{equation}
where we used the weighted arithmetic-geometric mean inequality. On account of \eqref{9}, we infer that
\begin{equation*}
		A_{i}^{-1}+{{M}^{-1}}{{m}^{-1}}{{M}^{\frac{{{m}^{-1}}I-A_{i}^{-1}}{{{m}^{-1}}-{{M}^{-1}}}}}{{m}^{\frac{A_{i}^{-1}-{{M}^{-1}I}}{{{m}^{-1}}-{{M}^{-1}}}}}\le ({{M}^{-1}}+{{m}^{-1}})I
		\end{equation*}
		whenever $mI\le {{A}_{i}}\le MI$ $\left( i=1,\ldots ,n \right)$. Multiplying the above inequality by ${{w}_{i}}$ and summing for all $i=1,\ldots ,n$ with $\sum\nolimits_{i=1}^{n}{{{w}_{i}}}=1$, we have
		\begin{equation}\label{10}
		\sum\limits_{i=1}^{n}{{{w}_{i}}A_{i}^{-1}}+{{M}^{-1}}{{m}^{-1}}\sum\limits_{i=1}^{n}{{{w}_{i}}{{M}^{\frac{{{m}^{-1}}I-A_{i}^{-1}}{{{m}^{-1}}-{{M}^{-1}}}}}{{m}^{\frac{A_{i}^{-1}-{{M}^{-1}I}}{{{m}^{-1}}-{{M}^{-1}}}}}}\le ({{M}^{-1}}+{{m}^{-1}})I.
		\end{equation}
		Then \eqref{10} implies
		\begin{equation}\label{14}
		\Phi \left( \sum\limits_{i=1}^{n}{{{w}_{i}}A_{i}^{-1}} \right)+{{M}^{-1}}{{m}^{-1}}\Phi \left( \sum\limits_{i=1}^{n}{{{w}_{i}}{{M}^{\frac{{{m}^{-1}}I-A_{i}^{-1}}{{{m}^{-1}}-{{M}^{-1}}}}}{{m}^{\frac{A_{i}^{-1}-{{M}^{-1}I}}{{{m}^{-1}}-{{M}^{-1}}}}}} \right)\le ({{M}^{-1}}+{{m}^{-1}})I
		\end{equation}
where $mI\le {{A}_{i}}\le MI\left( i=1,\ldots ,n \right)$. Now, we can write
\[\begin{aligned}
& {{M}^{-\frac{p}{2}}}{{m}^{-\frac{p}{2}}}\left\| \Phi {{\left( \sum\limits_{i=1}^{n}{{{w}_{i}}{{M}^{\frac{{{m}^{-1}}I-A_{i}^{-1}}{{{m}^{-1}}-{{M}^{-1}}}}}{{m}^{\frac{A_{i}^{-1}-{{M}^{-1}I}}{{{m}^{-1}}-{{M}^{-1}}}}}} \right)}^{\frac{p}{2}}}\Phi {{\left( {{\sigma}_{g}}\left( w;\mathbb{A} \right) \right)}^{-\frac{p}{2}}} \right\| \\ 
& \le \frac{1}{4}{{\left\| {{M}^{-\frac{p}{2}}}{{m}^{-\frac{p}{2}}}\Phi {{\left( \sum\limits_{i=1}^{n}{{{w}_{i}}{{M}^{\frac{{{m}^{-1}}I-A_{i}^{-1}}{{{m}^{-1}}-{{M}^{-1}}}}}{{m}^{\frac{A_{i}^{-1}-{{M}^{-1}I}}{{{m}^{-1}}-{{M}^{-1}}}}}} \right)}^{\frac{p}{2}}}+\Phi {{\left( {{\sigma}_{g}}\left( w;\mathbb{A} \right) \right)}^{-\frac{p}{2}}} \right\|}^{2}} \quad \text{(by \cite[Theorem 1]{6})}\\ 
& \le \frac{1}{4}{{\left\| {{\left( {{M}^{-1}}{{m}^{-1}}\Phi \left( \sum\limits_{i=1}^{n}{{{w}_{i}}{{M}^{\frac{{{m}^{-1}}I-A_{i}^{-1}}{{{m}^{-1}}-{{M}^{-1}}}}}{{m}^{\frac{A_{i}^{-1}-{{M}^{-1}I}}{{{m}^{-1}}-{{M}^{-1}}}}}} \right)+\Phi {{\left( {{\sigma}_{g}}\left( w;\mathbb{A} \right) \right)}^{-1}} \right)}^{\frac{p}{2}}} \right\|}^{2}} \quad \text{(by \cite[Corollary 3]{7})}\\ 
& =\frac{1}{4}{{\left\| {{M}^{-1}}{{m}^{-1}}\Phi \left( \sum\limits_{i=1}^{n}{{{w}_{i}}{{M}^{\frac{{{m}^{-1}}I-A_{i}^{-1}}{{{m}^{-1}}-{{M}^{-1}}}}}{{m}^{\frac{A_{i}^{-1}-{{M}^{-1}I}}{{{m}^{-1}}-{{M}^{-1}}}}}} \right)+\Phi {{\left( {{\sigma}_{g}}\left( w;\mathbb{A} \right) \right)}^{-1}} \right\|}^{p}} \\ 
& \le \frac{1}{4}{{\left\| {{M}^{-1}}{{m}^{-1}}\Phi \left( \sum\limits_{i=1}^{n}{{{w}_{i}}{{M}^{\frac{{{m}^{-1}}I-A_{i}^{-1}}{{{m}^{-1}}-{{M}^{-1}}}}}{{m}^{\frac{A_{i}^{-1}-{{M}^{-1}I}}{{{m}^{-1}}-{{M}^{-1}}}}}} \right)+\Phi \left( \sum\limits_{i=1}^{n}{{{w}_{i}}A_{i}^{-1}} \right) \right\|}^{p}} \quad \text{(by \eqref{6})}\\ 
& \le \frac{{{\left( {{M}^{-1}}+{{m}^{-1}} \right)}^{p}}}{4} \quad \text{(by \eqref{14})}.\\ 
\end{aligned}\]
Thus, we have shown
\[\left\| \Phi {{\left( \sum\limits_{i=1}^{n}{{{w}_{i}}{{M}^{\frac{{{m}^{-1}}I-A_{i}^{-1}}{{{m}^{-1}}-{{M}^{-1}}}}}{{m}^{\frac{A_{i}^{-1}-{{M}^{-1}I}}{{{m}^{-1}}-{{M}^{-1}}}}}} \right)}^{\frac{p}{2}}}\Phi {{\left( {{\sigma}_{g}}\left( w;\mathbb{A} \right) \right)}^{-\frac{p}{2}}} \right\|\le \frac{{{\left( {{M}^{-1}}+{{m}^{-1}} \right)}^{p}}}{4{{M}^{-\frac{p}{2}}}{{m}^{-\frac{p}{2}}}},\]
which is equivalent to the desired inequality \eqref{11}. \\
Now to prove \eqref{112nd}, we proceed similarly noting that for $t\in[m,M]$, we have
\begin{align}
t+mM\;(M^{-1})^{\frac{m^{-1}-t^{-1}}{m^{-1}-M^{-1}}}(m^{-1})^{\frac{t^{-1}-M^{-1}}{m^{-1}-M^{-1}}}&\leq t+mM\;t^{-1}\leq m+M.\label{needed_square_1}
\end{align}
Then

\begin{align}
& {{M}^{\frac{p}{2}}}{{m}^{\frac{p}{2}}}\left\| \Phi {{\left( \sum\limits_{i=1}^{n}{{{w}_{i}}{{M}^{\frac{{{m}^{-1}}I-A_{i}^{-1}}{{{m}^{-1}}-{{M}^{-1}}}}}{{m}^{\frac{A_{i}^{-1}-{{M}^{-1}I}}{{{m}^{-1}}-{{M}^{-1}}}}}} \right)}^{-\frac{p}{2}}}\Phi {{\left( {{\sigma}_{g}}\left( w;\mathbb{A} \right) \right)}^{\frac{p}{2}}} \right\| \notag\\ 
& \leq\frac{1}{4}{{\left\| {{M}}{{m}}\Phi \left( \sum\limits_{i=1}^{n}{{{w}_{i}}{{M}^{\frac{{{m}^{-1}}I-A_{i}^{-1}}{{{m}^{-1}}-{{M}^{-1}}}}}{{m}^{\frac{A_{i}^{-1}-{{M}^{-1}I}}{{{m}^{-1}}-{{M}^{-1}}}}}} \right)^{-1}+\Phi {{\left( {{\sigma}_{g}}\left( w;\mathbb{A} \right) \right)}} \right\|}^{p}}\notag \\ 
& \le \frac{1}{4}{{\left\| {{M}}{{m}}\Phi \left( \sum\limits_{i=1}^{n}{{{w}_{i}}{{(M^{-1})}^{\frac{{{m}^{-1}}I-A_{i}^{-1}}{{{m}^{-1}}-{{M}^{-1}}}}}{{(m^{-1})}^{\frac{A_{i}^{-1}-{{M}^{-1}I}}{{{m}^{-1}}-{{M}^{-1}}}}}} \right)+\Phi \left( \sum\limits_{i=1}^{n}{{{w}_{i}}A_{i}} \right) \right\|}^{p}}\label{needed_square_2}\\ 
& \notag \le \frac{{{\left( {{M}}+{{m}} \right)}^{p}}}{4},
\end{align}
where we have used  \eqref{6} and the fact that the function $f(t)=t^{-1}$ is operator convex to obtain \eqref{needed_square_2}, then \eqref{needed_square_1} to obtain the last inequality.

\end{proof}
As a complementary result to Theorem \ref{12} we have:
\begin{proposition}\label{prop_second_square}
	Let all assumptions of Theorem \ref{12} hold. Then
\begin{equation}\label{21}
\Phi {{\left( \sum\limits_{i=1}^{n}{{{w}_{i}}{{M}^{\frac{{{m}^{-1}}I-A_{i}^{-1}}{{{m}^{-1}}-{{M}^{-1}}}}}{{m}^{\frac{A_{i}^{-1}-{{M}^{-1}I}}{{{m}^{-1}}-{{M}^{-1}}}}}} \right)}^{p}}\le \left( \frac{\left( m+M \right)^{2}}{4^{\frac{2}{p}}{mM}} \right)^{p}{{\sigma}_{g}}{{\left( w;\Phi \left( \mathbb{A} \right) \right)}^{p}},
\end{equation}
and

\begin{equation}\label{212nd}
{{\sigma}_{g}}{{\left( w;\Phi \left( \mathbb{A} \right) \right)}^{p}}\le \left( \frac{\left( m+M \right)^{2}}{4^{\frac{2}{p}}{mM}} \right)^{p}\;\Phi {{\left( \sum\limits_{i=1}^{n}{{{w}_{i}}{{M}^{\frac{{{m}^{-1}}I-A_{i}^{-1}}{{{m}^{-1}}-{{M}^{-1}}}}}{{m}^{\frac{A_{i}^{-1}-{{M}^{-1}I}}{{{m}^{-1}}-{{M}^{-1}}}}}} \right)}^{p}}.
\end{equation}

\end{proposition}
\begin{proof}
We prove the first inequality. Notice that
\[\begin{aligned}
& {{M}^{-\frac{p}{2}}}{{m}^{-\frac{p}{2}}}\left\| \Phi {{\left( \sum\limits_{i=1}^{n}{{{w}_{i}}{{M}^{\frac{{{m}^{-1}}I-A_{i}^{-1}}{{{m}^{-1}}-{{M}^{-1}}}}}{{m}^{\frac{A_{i}^{-1}-{{M}^{-1}I}}{{{m}^{-1}}-{{M}^{-1}}}}}} \right)}^{\frac{p}{2}}}{{\sigma}_{g}}{{\left( w;\Phi \left( \mathbb{A} \right) \right)}^{-\frac{p}{2}}} \right\| \\ 
& \le \frac{1}{4}{{\left\| {{M}^{-\frac{p}{2}}}{{m}^{-\frac{p}{2}}}\Phi {{\left( \sum\limits_{i=1}^{n}{{{w}_{i}}{{M}^{\frac{{{m}^{-1}}I-A_{i}^{-1}}{{{m}^{-1}}-{{M}^{-1}}}}}{{m}^{\frac{A_{i}^{-1}-{{M}^{-1}I}}{{{m}^{-1}}-{{M}^{-1}}}}}} \right)}^{\frac{p}{2}}}+{{\sigma}_{g}}{{\left( w;\Phi \left( \mathbb{A} \right) \right)}^{-\frac{p}{2}}} \right\|}^{2}} \\ 
& \le \frac{1}{4}{{\left\| {{\left( {{M}^{-1}}{{m}^{-1}}\Phi \left( \sum\limits_{i=1}^{n}{{{w}_{i}}{{M}^{\frac{{{m}^{-1}}I-A_{i}^{-1}}{{{m}^{-1}}-{{M}^{-1}}}}}{{m}^{\frac{A_{i}^{-1}-{{M}^{-1}I}}{{{m}^{-1}}-{{M}^{-1}}}}}} \right)+{{\sigma}_{g}}{{\left( w;\Phi \left( \mathbb{A} \right) \right)}^{-1}} \right)}^{\frac{p}{2}}} \right\|}^{2}} \\ 
& =\frac{1}{4}{{\left\| {{M}^{-1}}{{m}^{-1}}\Phi \left( \sum\limits_{i=1}^{n}{{{w}_{i}}{{M}^{\frac{{{m}^{-1}}I-A_{i}^{-1}}{{{m}^{-1}}-{{M}^{-1}}}}}{{m}^{\frac{A_{i}^{-1}-{{M}^{-1}I}}{{{m}^{-1}}-{{M}^{-1}}}}}} \right)+{{\sigma}_{g}}{{\left( w;\Phi \left( \mathbb{A} \right) \right)}^{-1}} \right\|}^{p}} \\ 
& \le \frac{1}{4}{{\left\| {{M}^{-1}}{{m}^{-1}}\Phi \left( \sum\limits_{i=1}^{n}{{{w}_{i}}{{M}^{\frac{{{m}^{-1}}I-A_{i}^{-1}}{{{m}^{-1}}-{{M}^{-1}}}}}{{m}^{\frac{A_{i}^{-1}-{{M}^{-1}I}}{{{m}^{-1}}-{{M}^{-1}}}}}} \right)+\sum\limits_{i=1}^{n}{{{w}_{i}}\Phi {{\left( {{A}_{i}} \right)}^{-1}}} \right\|}^{p}} \quad \text{(by LHS of \eqref{3})}\\ 
& \le \frac{1}{4}{{\left\| {{M}^{-1}}{{m}^{-1}}\Phi \left( \sum\limits_{i=1}^{n}{{{w}_{i}}{{M}^{\frac{{{m}^{-1}}I-A_{i}^{-1}}{{{m}^{-1}}-{{M}^{-1}}}}}{{m}^{\frac{A_{i}^{-1}-{{M}^{-1}I}}{{{m}^{-1}}-{{M}^{-1}}}}}} \right)+\Phi \left( \sum\limits_{i=1}^{n}{{{w}_{i}}A_{i}^{-1}} \right) \right\|}^{p}} \quad \text{(by \cite[Theorem 2.3.6]{8})}\\ 
& \le \frac{{{\left( {{M}^{-1}}+{{m}^{-1}} \right)}^{p}}}{4} \quad \text{(by \eqref{14})}.\\ 
\end{aligned}\]
Then the desired inequality follows immediately. The second inequality follows similarly.
\end{proof}

Following the same steps as Theorem \ref{12} and Proposition \ref{prop_second_square}, one can show that the inequalities in \ref{thm_first_linear_map} can be raised to the $p^{\text{th}}$ power.
\begin{corollary}
Let $\mathbb{A}=\left( {{A}_{1}},\ldots ,{{A}_{n}} \right)$ be an $n$-tuple of positive definite matrices in $\mathcal{M}_k$  with $mI\le {{A}_{i}}\le MI\left( i=1,\ldots ,n \right)$ for some scalars $0<m<M$ and $w=\left( {{w}_{1}},\ldots ,{{w}_{n}} \right)$ be a weight vector. Then for any unital positive linear map $\Phi$ and $p\geq 2,$
	$$\sigma_g(w;\Phi(\mathbb{A}))^p\leq \left(\frac{(M+m)^2}{4^{\frac{2}{p}}mM}\right)^{p}\Phi(\sigma_g(w;\mathbb{A}))^p.$$
\end{corollary}

Related to positive linear maps, the inequality \eqref{bourin_geo_ineq} was shown  in \cite{bourin}  as a main tool to prove a reversed version of the inequality $\Phi(A\sharp B)\leq \Phi(A)\sharp\Phi(B).$ Our next result is the natural extension of \eqref{bourin_geo_ineq} to the context of the solution of the GKE. 
\begin{theorem}
Let $\mathbb{A}=(A_1,\cdots,A_n)$ be an $n$-tuple of positive definite matrices in $\mathcal{M}_k$ and let $w=(w_1,\cdots,w_n)$ be a weight vector. Then,  for any  $x\in\mathbb{C}^k$,
 \begin{align*}
\left<\sigma_g(w;\mathbb{A})x,x\right>\leq \sigma_g\left(w;\left<\mathbb{A}x,x\right>\right).
\end{align*}
\end{theorem}

\begin{proof}
If $x=0,$ the result holds trivially, hence we may assume $x\not=0$.
Let $X=\sigma_g(w;\mathbb{A}).$ Then
\[\sum\limits_{i=1}^{n}{{{w}_{i}}g\left( {{X}^{-\frac{1}{2}}}{{A}_{i}}{{X}^{-\frac{1}{2}}} \right)}=0\quad\Rightarrow\quad \sum\limits_{i=1}^{n}{{{w}_{i}}{{X}^{\frac{1}{2}}}g\left( {{X}^{-\frac{1}{2}}}{{A}_{i}}{{X}^{-\frac{1}{2}}} \right)}{{X}^{\frac{1}{2}}}=0.\]
Therefore, if $x\in\mathbb{C}^n$ is any nonzero vector, so that $X x\not=0,$ then
\[\begin{aligned}
 0=&\sum\limits_{i=1}^{n}{{{w}_{i}}\left\langle {{X}^{\frac{1}{2}}}g\left( {{X}^{-\frac{1}{2}}}{{A}_{i}}{{X}^{-\frac{1}{2}}} \right){{X}^{\frac{1}{2}}}x,x \right\rangle } \\ 
& ={{\left\| {{X}^{\frac{1}{2}}}x \right\|}^{2}}\sum\limits_{i=1}^{n}{{{w}_{i}}\left\langle g\left( {{X}^{-\frac{1}{2}}}{{A}_{i}}{{X}^{-\frac{1}{2}}} \right)\frac{{{X}^{\frac{1}{2}}}x}{\left\| {{X}^{\frac{1}{2}}}x \right\|},\frac{{{X}^{\frac{1}{2}}}x}{\left\| {{X}^{\frac{1}{2}}}x \right\|} \right\rangle } \\ 
& \le {{\left\| {{X}^{\frac{1}{2}}}x \right\|}^{2}}\sum\limits_{i=1}^{n}{{{w}_{i}}g\left( \left\langle {{X}^{-\frac{1}{2}}}{{A}_{i}}{{X}^{-\frac{1}{2}}}\frac{{{X}^{\frac{1}{2}}}x}{\left\| {{X}^{\frac{1}{2}}}x \right\|},\frac{{{X}^{\frac{1}{2}}}x}{\left\| {{X}^{\frac{1}{2}}}x \right\|} \right\rangle  \right)} \\ 
& ={{\left\| {{X}^{\frac{1}{2}}}x \right\|}^{2}}\sum\limits_{i=1}^{n}{{{w}_{i}}g\left( \frac{\left\langle {{A}_{i}}x,x \right\rangle }{{{\left\| {{X}^{\frac{1}{2}}}x \right\|}^{2}}} \right)}.  
\end{aligned}\]
Applying Lemma \ref{4.1}, we infer that
\begin{align*}
\sigma_g\left(w;\frac{\left<\mathbb{A}x,x\right>}{\left\|X^{\frac{1}{2}}x\right\|^2}\right)\geq 1\quad\Rightarrow\quad \left\|X^{\frac{1}{2}}x\right\|^{-2}\;\sigma_g\left(w;\left<\mathbb{A}x,x\right>\right)\geq 1,
\end{align*}
which implies
\begin{align*}
\left<X^{\frac{1}{2}}x,X^{\frac{1}{2}}x\right>\leq \sigma_g\left(w;\left<\mathbb{A}x,x\right>\right)\quad\Rightarrow\quad \left<Xx,x\right>\leq \sigma_g\left(w;\left<\mathbb{A}x,x\right>\right).
\end{align*}
That is
\begin{align*}
\left<\sigma_g(w;\mathbb{A})x,x\right>\leq \sigma_g\left(w;\left<\mathbb{A}x,x\right>\right).
\end{align*}
\end{proof}

\section{Inequalities for operator monotone functions}
Given an operator monotone function $f:(0,\infty)\to (0,\infty)$, we  discuss the relation between ${{\sigma }_{g}}\left( w;f\left( \mathbb{A} \right) \right)$ and $f\left( {{\sigma }_{g}}\left( w;\mathbb{A} \right) \right)$ in the following theorem. 
\begin{theorem}\label{thm_monotone}
	Let $\mathbb{A}=\left( {{A}_{1}},\ldots ,{{A}_{n}} \right)$ be a $n$-tuple of positive definite matrices in $\mathcal{M}_k$ with $mI\le {{A}_{i}}\le MI\left( i=1,\ldots ,n \right)$ for some scalars $0<m<M$ and $w=\left( {{w}_{1}},\ldots ,{{w}_{n}} \right)$ be a weight vector.
	\begin{itemize}
		\item[(I)] If $f:(0,\infty)\to (0,\infty)$ is an operator monotone function, then
		\begin{equation}\label{16}
		{{\sigma }_{g}}\left( w;f\left( \mathbb{A} \right) \right)\le K\left( h,2 \right)f\left( {{\sigma }_{g}}\left( w;\mathbb{A} \right) \right).
		\end{equation}
		\item[(II)] If $f:(0,\infty)\to (0,\infty)$ is an operator monotone decreasing function, then
		\begin{equation}\label{17}
		f\left( {{\sigma }_{g}}\left( w;\mathbb{A} \right) \right)\le K\left( h,2 \right){{\sigma }_{g}}\left( w;f\left( \mathbb{A} \right) \right).	
		\end{equation}
	\end{itemize}
\end{theorem}
\begin{proof}
Notice first that if $f:(0,\infty)\to (0,\infty)$ is operator monotone increasing, it is operator concave \cite{ando}. 
	Compute
	\[\begin{aligned}
	{{\sigma }_{g}}\left( w;f\left( \mathbb{A} \right) \right)&\le \sum\limits_{i=1}^{n}{{{w}_{i}}f\left( {{A}_{i}} \right)} \quad({\text{by}}\;\eqref{6})\\ 
	& \le f\left( \sum\limits_{i=1}^{n}{{{w}_{i}}{{A}_{i}}} \right)\quad (f\;{\text{being\;operator\;concave}}) \\ 
	& \le f\left( K\left( h,2 \right){{\sigma }_{g}}\left( w;\mathbb{A} \right) \right) \quad ({\text{by}}\;\eqref{03})\\ 
	& \le K\left( h,2 \right)f\left( {{\sigma }_{g}}\left( w;\mathbb{A} \right) \right)  
	\end{aligned}\]
	where, to obtain the last inequality, we have used  the fact that if $f\left( t \right)$ is operator monotone and $\alpha \ge 1$, then $f\left( \alpha t \right)\le \alpha f\left( t \right)$. This proves the first inequality.
	
	For the second inequality, notice first that when $f$ is operator monotone decreasing, we have
	$\frac{1}{\alpha }f\left( t \right)\le f\left( \alpha t \right)$ for $\alpha\geq 1.$ Then
	\[\begin{aligned}
	\frac{1}{K\left( h,2 \right)}f\left( {{\sigma }_{g}}\left( w;\mathbb{A} \right) \right)&\le f\left( K\left( h,2 \right){{\sigma }_{g}}\left( w;\mathbb{A} \right) \right) \\ 
	& \le f\left( \sum\limits_{i=1}^{n}{{{w}_{i}}{{A}_{i}}} \right)\quad({\text{by}}\;\eqref{03})\\ 
	& \le {{\left( \sum\limits_{i=1}^{n}{{{w}_{i}}f{{\left( {{A}_{i}} \right)}^{-1}}} \right)}^{-1}} \\ 
	& \le {{\sigma }_{g}}\left( w;f\left( \mathbb{A} \right) \right)  \quad ({\text{by}}\;\eqref{6})
	\end{aligned}\]
	where we have used the fact that $f$ is operator monotone decreasing to obtain the second inequality and \cite[Remark 2.7]{ando} to obtain the third inequality.
\end{proof}
As a counterpart of \eqref{16}, we have the following reversed version. The notation $\nabla_n\mathbb{A}$ will be used for the arithmetic mean $\frac{1}{n}(A_1+\cdots+A_n).$
\begin{proposition}
	Let $\mathbb{A}=\left( {{A}_{1}},\ldots ,{{A}_{n}} \right)$ be a $n$-tuple of positive definite matrices in $\mathcal{M}_k$ with $mI\le {{A}_{i}}\le MI\left( i=1,\ldots ,n \right)$ for some scalars $0<m<M$ and $w=\left( {{w}_{1}},\ldots ,{{w}_{n}} \right)$ a weight vector.  	If $f:(0,\infty)\to (0,\infty)$ is an operator monotone function, then
	\begin{equation}\label{160}
	f\left( {{\sigma }_{g}}\left( w;\mathbb{A} \right) \right)\le K\left( h,2 \right){{\sigma }_{g}}\left( w;f\left( \mathbb{A} \right) \right)+n{{w}_{\max }}\left( f\left( {{\nabla }_{n}}\mathbb{A} \right)-{{\nabla }_{n}}f\left( \mathbb{A} \right) \right),
	\end{equation}
	where $w_{\text{max}}=\max\limits_{1\leq i\leq n}w_i$ and $h=\frac{f(M)}{f(m)}.$
\end{proposition}
\begin{proof}
	Noting operator concavity and monotonicity of $f$, we have
	\[\begin{aligned}
	f\left( {{\sigma }_{g}}\left( w;\mathbb{A} \right) \right)&\le f\left( \sum\limits_{i=1}^{n}{{{w}_{i}}{{A}_{i}}} \right) \\ 
	& \le \sum\limits_{i=1}^{n}{{{w}_{i}}f\left( {{A}_{i}} \right)}+n{{w}_{\max }}\left( f\left( {{\nabla }_{n}}\mathbb{A} \right)-{{\nabla }_{n}}f\left( \mathbb{A} \right) \right) \\ 
	& \le K\left( h,2 \right){{\sigma }_{g}}\left( w;f\left( \mathbb{A} \right) \right)+n{{w}_{\max }}\left( f\left( {{\nabla }_{n}}\mathbb{A} \right)-{{\nabla }_{n}}f\left( \mathbb{A} \right) \right)  
	\end{aligned}\]
	where, to obtain the second inequality, we have used an argument similar to that used in \cite{mitroi}.
\end{proof}
Notice that the function $f(t)=t^{-1}$ is operator monotone decreasing. Therefore, inequality \eqref{17} implies
$$\sigma_{g}(w;\mathbb{A})^{-1}\leq K(h,2)\;\sigma_g(w;\mathbb{A}^{-1}).$$ In the following result, we present a counterpart of this inequality. \\
For this, notice that when $g(x)=x-1$, we have $\sigma_g(w;\mathbb{A})=\sum_{i=1}^{n}w_iA_i$, while we obtain $\sigma_g(w,A)=\left(\sum_{i=1}^{n}A_i^{-1}\right)^{-1}$ when $g(x)=1-x^{-1}.$ Therefore, letting $g(x)=x-1$ in \eqref{04}, we have
\begin{equation}\label{needed_1}
\sum_{i=1}^{n}w_iA_i\leq K(h,2)\left(\sum_{i=1}^{n}w_iA_i^{-1}\right)^{-1}.
\end{equation}
\begin{proposition}
	Let all assumptions of Theorem \ref{thm_monotone} hold. Then
	\begin{equation}\label{ineq_inverse}
	\sigma_g(w;\mathbb{A}^{-1})\leq K(h,2)\;\sigma_g(w;A)^{-1}.
	\end{equation}
\end{proposition}
\begin{proof}
	Notice that
	\begin{align*}
	\sigma_g(w;\mathbb{A}^{-1})&\leq \sum_{i=1}^{n}w_iA_i^{-1}\quad({\text{by}}\;\eqref{6})\\
	&\leq K(h,2)\left(\sum_{i=1}^{n}w_iA_i\right)^{-1}\quad({\text{by}}\;\eqref{needed_1})\\
	&\leq K(h,2)\;\sigma_g(w;\mathbb{A})^{-1}\quad({\text{by}}\;\eqref{6}).
	\end{align*}
\end{proof}

\textbf{Acknowledgment:} The authors would like to thank the referees for their useful comments.

\vskip 0.3 true cm

{\tiny {(M. Sababheh) Department of Basic Sciences, Princess Sumaya University for Technology, Amman,
Jordan.}}

{\tiny \textit{E-mail address:} sababheh@yahoo.com, sababheh@psut.edu.jo}

{\tiny \vskip 0.3 true cm }

{\tiny (H.R. Moradi) Young Researchers and Elite Club, Mashhad Branch, Islamic Azad
	University, Mashhad, Iran. }

{\tiny \textit{E-mail address:} hrmoradi@mshdiau.ac.ir}
	
{\tiny \vskip 0.3 true cm }

{\tiny (Z. Heydarbeygi) Department of Mathematics, Mashhad Branch, Islamic Azad
	University, Mashhad, Iran.}

{\tiny \textit{E-mail address:} zheydarbeygi@yahoo.com}
	
%-----------------------------------------------------------------------------
%-----------------------------------------------------------------------------
\end{document}